\documentclass[10pt,a4paper]{amsart}
\usepackage{amsfonts}
\usepackage[latin9]{inputenc}
\usepackage{amsmath}
\usepackage{amssymb}
\usepackage{breqn}
\usepackage{url}
\usepackage{graphicx}
\usepackage[pdfstartview=FitH]{hyperref}
\usepackage{amsthm}
\usepackage{hyperref}
\usepackage{url}
\usepackage{verbatim}
\usepackage{enumitem}
\usepackage{appendix}
\usepackage{cleveref}
\usepackage{bbm}
\newcommand{\Ind}{\operatorname{Ind}}
\newcommand{\sgn}{\operatorname{sgn}}
\newcommand{\Ker}{\operatorname{Ker}}
\newcommand{\im}{\operatorname{Im}}
\newcommand{\tr}{\operatorname{tr}}
\newcommand{\id}{\operatorname{id}}

\crefformat{section}{\S#2#1#3} 
\crefformat{subsection}{\S#2#1#3}
\crefformat{subsubsection}{\S#2#1#3}

\makeatletter
\begin{document}
\newtheorem{theorem}{Theorem}[section]
\newtheorem{lemma}[theorem]{Lemma}
\newtheorem{definition}[theorem]{Definition}
\newtheorem{claim}[theorem]{Claim}
\newtheorem{example}[theorem]{Example}
\newtheorem{remark}[theorem]{Remark}
\newtheorem{proposition}[theorem]{Proposition}
\newtheorem{corollary}[theorem]{Corollary}
\newtheorem{observation}[theorem]{Observation}

\author{Alexander Lubotzky}
\address{Einstein Institute of Mathematics, The Hebrew University
of Jerusalem, 91904, Jerusalem, Israel}
\email{alex.lubotzky@mail.huji.ac.il}

\author{Izhar Oppenheim}
\address{Department of Mathematics, Ben-Gurion University of the Negev, Be'er Sheva 84105, Israel} 
\email{izharo@bgu.ac.il}


\title{Non $p$-norm approximated Groups}
\maketitle
\begin{abstract}
It was shown in a previous work of the first named author with De Chiffre, Glebsky and Thom that there exists a finitely presented group which cannot be approximated by almost-homomorphisms to the unitary groups $U(n)$ equipped with the Frobenius norms (a.k.a as $L^2$ norm, or the Schatten-2-norm). In his ICM18 lecture, Andreas Thom asks if this result can be extended to general Schatten-p-norms. We show that this is indeed the case for $1<p< \infty$. 
\end{abstract}

\section{Introduction}

Let $U(n)$ be the group of unitary $n \times n$ matrices equipped with a bi-invariant metric $d_n$ induced by a Banach norm $\Vert . \Vert$ on $M_n (\mathbb{C})$, as $d_n (g,h) = \Vert g - h \Vert$. Examples of special interest are:
\begin{enumerate}[label=({{\arabic*}})]
\item The Hilbert-Schmidt norm: $\Vert T \Vert_{H.S.} = \sqrt{\frac{1}{n} \tr (T^* T)}$.
\item For $1 \leq p < \infty$, the Schatten $p$-norm: $\Vert T \Vert_p =  \left( \tr \vert T \vert^p \right)^{\frac{1}{p}}$, where $\vert T \vert = \sqrt{T^* T}$. When $p=2$, this is usually called the Frobenius norm: 
$$\Vert T \Vert_2 = \Vert T \Vert_{Frob} = \sqrt{n} \Vert T \Vert_{H.S.}.$$
\item The operator norm, $\Vert T \Vert_{op} = \max \lbrace \Vert Tv \Vert : \Vert v \Vert=1 \rbrace$ also known as the Schatten $\infty$-norm.
\end{enumerate}

Whatever $\lbrace d_n \rbrace_{n=1}^\infty$ are, define for $\mathcal{G} = (U(n), d_n)$ the following:
\begin{definition}
A finitely presented group $\Gamma$ is called $\mathcal{G}$-approximated if there exists an infinite sequence $\lbrace n_k \rbrace_{k=1}^\infty$ of integers and (set-theoretic) maps $\phi = (\phi_{n_k})$, $\phi_{n_k} : \Gamma \rightarrow U(n_k)$ such that:
\begin{enumerate}
\item $\forall g, h \in \Gamma$, $\lim d_{n_k} (\phi_{n_k} (gh), \phi_{n_k} (g)\phi_{n_k} (h)) = 0$.
\item $\forall g \in \Gamma, g \neq 1$, there is $\varepsilon (g) = \varepsilon >0$ such that $\limsup d_{n_k} (\phi_{n_k} (g), \id_{U(n_k)}) \geq \varepsilon$, where $\id_{U(n_k)}$ is the $n_k \times n_k$ identity matrix.  
\end{enumerate}
\end{definition}

There are two long standing questions regarding whether there exist groups $\Gamma$ which are \underline{not} $(U(n), d_n)$-approximated with respect to the $d_n$'s defined in cases $(1)$ and $(3)$. The question for case $(1)$ where $d_n$ is defined by the Hilbert-Schmidt norm, is equivalent to Alain Connes' problem whether every group is Connes-embeddable (see \cite{Connes} and \cite{Pestov}  for details), while case $(3)$ is related to Kirchberg's question whether any stably finite $C^*$-algebra is embeddable into an norm-ultraproduct of matrix algebras (see \cite{BlackK} for details), which implies that any group is $(U(n), d_n)$-approximated with respect to the distance induced by the operator norm. 

In this paper, a group $\Gamma$ will be called $p$-norm approximated if it is approximated with respect to $\mathcal{G} = (U(n), \Vert . \Vert_p)$.

A recent breakthrough \cite{DGLT} shows that there exist groups that are not Frobenius approximated (i.e.,  groups that are not $2$-norm approximated). Following this, Andreas Thom asks in his ICM 2018 talk \cite{Thom}, if that result can be extended to all Schatten $p$-norms. We answer this affirmatively in the case where $1 < p < \infty$, and in fact we prove a somewhat stronger result:
\begin{theorem}
\label{main thm}
There exists a finitely presented group $\Lambda$ which is not $p$-norm approximated for any $1 < p < \infty$. 
\end{theorem}
The case of $p=1$ is left open, as well as the cases of the Hilbert-Schmidt and the operator norms.

The method of proof follows the one implemented in \cite{DGLT} for $p=2$, but some further cohomology vanishing results are needed. 

Let $\Gamma$ be a finitely presented group $\Gamma = \langle S \vert R \rangle$, with $R \subseteq \mathbb{F}_S$ - the free group on $S$ and $\vert R \vert < \infty$. Any map $\phi : S \rightarrow U(n)$ uniquely determines a homomorphism $\phi :\mathbb{F}_S \rightarrow U(n)$ which we will also denote by $\phi$. 

The group $\Gamma$  is called $\mathcal{G}=(U(n), d_n)$-stable if for every $\varepsilon >0$ there exists $\delta >0$ such that for every $n \in \mathbb{N}$, if $\phi : S \rightarrow U(n)$ is a map with 
$$\sum_{r \in R} d_n (\phi (r), \id_{U(n)} ) < \delta,$$ 
then there exists a homomorphism $\tilde{\phi}: \Gamma \rightarrow U(n)$ (or equivalently, a map $\tilde{\phi}  : S \rightarrow U(n)$ with $\sum_{r \in R} d_n (\tilde{\phi}  (r), \id_{U(n)} ) = 0$) with 
$$\sum_{s \in S} d_n (\phi (s), \tilde{\phi}  (s)) < \varepsilon.$$

Below, we will call a group $\Gamma$ \textit{$p$-norm stable} if it is stable with respect to  $\mathcal{G} = (U(n), \Vert . \Vert_p)$.

A well-known observation (see for instance \cite{GR}, \cite{AP} and \cite{DGLT}) is that a $\mathcal{G}$-approximated $\mathcal{G}$-stable finitely presented group must be residually finite. Thus a non-residually-finite finitely presented group which is $\mathcal{G}$-stable cannot be $\mathcal{G}$-approximated.

In \cite{DGLT}, a general sufficient criterion for Frobenius stability was given: If $H^2 (\Gamma,V) =0$ for every unitary representation of $\Gamma$ on any Hilbert space $V$, then $\Gamma$ is Frobenius stable. This was combined then with Garland's method \cite{Gar} (as extended by Ballmann and \' Swiatkowski \cite{BS} for general Hilbert spaces) to produce some lattices $\Gamma_0$ in some simple $l$-adic Lie groups satisfying the desired $H^2$ vanishing for every Hilbert space. Then a $l$-adic analogue of a result by Deligne \cite{Del} was implemented in order to produce some finite central extensions $\widetilde{\Gamma}$ of $\Gamma_0$ that are not residually finite. These $\widetilde{\Gamma}$ are the non Frobenius approximated groups.

The proof in \cite{DGLT} actually shows more (see Theorem 5.1 and Remark 5.2 there): If $\Vert . \Vert$ is any \underline{unitarily invariant} and \underline{submultiplicative} norm on $M_n (\mathbb{C})$ (and so is the Schatten $p$-norm for every $1 \leq p \leq \infty$) and if $H^2 (\Gamma, V)=0$ for any $\Gamma$ isometric representation on a Banach space of the form $V = \prod_{k \rightarrow \mathcal{U}} (M_{n_k} (\mathbb{C}), \Vert . \Vert)$, where $V$ is the Banach ultraproduct of $M_{n_k} (\mathbb{C})$ with respect to the norm $\Vert . \Vert$ and with respect to any ultrafilter $\mathcal{U}$ (see \cite{DGLT} and \cref{Preliminaries sec} below for more) then $\Gamma$ is $\mathcal{G}$-stable. To get non $p$-norm approximated groups we need an $H^2$-vanishing result which will work for spaces of the form $V = \prod_{k \rightarrow \mathcal{U}} (M_{n_k} (\mathbb{C}), \Vert . \Vert_p)$, where $\Vert . \Vert_p$ is the Schatten $p$-norm. 

The technology to extend Garland's method (or more precisly the method of Dymara and Januszkiewicz \cite{DymaraJ}) to a wide class of Banach spaces was developed by the second named author in \cite{OppVanBan}. More precisely, it is shown there that for certain classes of Banach spaces, vanishing of cohomology can be deduced for a $l$-adic Lie group $G$, given a large enough thickness of the affine building on which it acts. Using (a suitable version of) Shapiro's Lemma, these vanishing results pass to cocompact lattices of $G$. In our context, these methods yield the following theorem: 
\begin{theorem}
\label{Vanishing of cohomoloogy thm}
Let $G= \mathbb{G}(K)$, where $\mathbb{G}$ is a simple $K$-algebraic group of $K$-rank $d$ over a non-archimedean local field $K$ with residue field of order $q$ and $\Gamma_0 < G$ a cocompact lattice of $G$. For any $1 < p_1 \leq 2 \leq p_2 < \infty$, there exists a natural number $Q=Q(p_1,p_2,d)$ such that if $q > Q$, then $H^i  (\Gamma_0, V) =0$ for every $i=1,...,d-1$ and every Banach space of the form $V = \prod_{k \rightarrow \mathcal{U}} (M_{n_k} (\mathbb{C}), \Vert . \Vert_p)$ where $\mathcal{U}$ is any ultrafilter on $\mathbb{N}$ and $p_1 \leq p \leq p_2$.
\end{theorem}

Most of the paper will be devoted to the proof of Theorem \ref{Vanishing of cohomoloogy thm}. Let us now show how it implies Theorem \ref{main thm}. 

Applying Deligne's method as in \cite{DGLT}, we get a non-residually finite, finite central extension $\widetilde{\Gamma} = \Gamma_{p_1,p_2}$ of such a cocompact lattice $\Gamma_0$ in a suitable $l$-adic Lie group $G$. Assuming that the dimension of the affine building associated to $G$ is greater or equal to $3$, a standard spectral sequence argument yields that $H^2  (\widetilde{\Gamma}, V) =0$ for any $V$ as in Theorem \ref{Vanishing of cohomoloogy thm}. Therefore $\widetilde{\Gamma}$ is $p$-norm stable for any $p_1 \leq p \leq p_2$ and since it is not residually finite, we deduce by the observation stated above that  $\widetilde{\Gamma}$ is \textbf{not} $p$-norm approximated for any $p_1 \leq p \leq p_2$.

Recalling now Higman's Theorem (see \cite[Theorem 7.3, page 215]{LSBook}) which asserts that there exists a finitely presented group $\Lambda$ that contains all finitely presented groups. By taking $p_1 \rightarrow 1$ and $p_2 \rightarrow \infty$ and noting that if a group is $\mathcal{G}$-approximated, so is every subgroup of it, we deduce that such $\Lambda$ is not $p$-norm approximated for any $1 < p < \infty$ and Theorem \ref{main thm} is proved. 

As mentioned above, the cases of $p=1$ and $p= \infty$ are left open. In both cases (unlike the Hilbert-Schmidt norm) the norms are submultiplicative  (see \cite{DGLT} for an explanation of the importance of this property), but at least for $p=\infty$ (which is the case of the operator norm), the method of this paper cannot work: the method applied below shows vanishing of $H^i  (\Gamma_0, V) =0$ for every $i=1,...,d-1$, based on the geometric properties of $V$. Therefore, if the vanishing of cohomology is proved for $p=\infty$, it will be proven for every $\ell^\infty$ Banach space, but it is known that for every discrete group $\Gamma$, $H^1 (\Gamma, \ell^\infty (\Gamma)) \neq 0$ (see for instance \cite[Section 4]{DGLT}). We note that this type of reasoning excluding $p=\infty$ does not hold in the case of $p=1$: in \cite{BGM}, Bader, Gelander and Monod showed that for every group $\Gamma$ with property (T), $H^1 (\Gamma, L^1 (\Omega)) =0$ for every measure space $\Omega$. The methods of \cite{BGM} are very different from those applied in this paper (and in \cite{OppAvProj}), but one can ask it those methods can be extended to show the vanishing of the second cohomology for the case $p=1$. 

The rest of this paper is devoted to the proof of Theorem \ref{Vanishing of cohomoloogy thm}. As noted above, in \cite{OppVanBan}, the second named author proved a similar vanishing of cohomology, but for a $l$-adic Lie group $G$ instead of a lattice. Below, we will show how to use the results of \cite{OppVanBan} together with a version of Shapiro's Lemma to deduce Theorem \ref{Vanishing of cohomoloogy thm}. The paper \cite{OppVanBan} was not written with this application in mind and therefore in order to adapt the results of \cite{OppVanBan} to our setting, a somewhat lengthy exposition regarding the general theory of Banach spaces and group representations on them is needed. 

This paper is organized as follows: In \cref{Preliminaries sec}, we give a number of definitions and results needed to state the results regarding group cohomology with Banach coefficients. In \cref{Vanishing of Cohomology sec}, we deduce Theorem \ref{Vanishing of cohomoloogy thm} from the results of \cite{OppVanBan} that apply to $G$, using Shapiro's Lemma that relates the cohomology of $G$ to that of $\Gamma_0$. Unfortunately, it seems that the version of Shapiro's Lemma we need (for Banach spaces rather than Hilbert spaces) is not proved in the literature and therefore we will provide a proof in \cref{Shapiro's lemma sec}.

\paragraph{\textbf{Acknowledgments.}} The first named author was supported in part by the ERC and the NSF. The second named author was supported in part by the ISF. This work was done while the authors were visiting the IIAS (Israeli Institute of Advanced Studies) whose great hospitality is warmly acknowledged.

\section{Preliminaries}
\label{Preliminaries sec}

\subsection{Strictly $\theta$-Hilbertian spaces and Schatten norms}

Two Banach spaces $V_0, V_1$ form a \textit{compatible pair} $(V_0,V_1)$ if they are continuously linearly embedded in the same topological vector space. The idea of complex interpolation is that given a compatible pair $(V_0,V_1)$ and a constant $0 \leq \theta \leq 1$, there is a method to produce a new Banach space $[V_0, V_1]_\theta$ as a ``combination'' of $V_0$ and $V_1$. We will not review this method here, and the interested reader can find more information on interpolation in \cite{InterpolationSpaces}.

This brings us to consider the following definition due to Pisier \cite{Pisier}: a Banach space $V$ is called \textit{strictly $\theta$-Hilbertian} for $0 < \theta \leq 1$, if there is a compatible pair $(V_0,V_1)$, with $V_1$ a Hilbert space, such that $V = [V_0, V_1]_\theta$. Examples of strictly $\theta$-Hilbertian spaces are $L^p$ space and non-commutative $L^p$ spaces (see \cite{PX} for definitions and properties of non-commutative $L^p$ spaces), where in these cases $\theta = \frac{2}{p}$ if $2 \leq  p  < \infty $ and $\theta = 2 - \frac{2}{p}$ if $1 < p \leq 2$. We are interested in a very basic case of non-commutative $L^p$ spaces - namely finite matrices with $p$-Schatten norms:
\begin{definition}[Schatten norm for matrices]
Let $d \in \mathbb{N}$ and let $M_d (\mathbb{C})$ be the space of $d \times d$ complex matrices. For $A \in M_d (\mathbb{C})$, recall that $A^* A$ is always a positive semidefinite matrix and denote $\vert A \vert = \sqrt{A^* A}$. For $1 \leq p < \infty$, define the Schatten $p$-norm on $M_d (\mathbb{C})$ by $\Vert A \Vert_p = \left( \tr(\vert A \vert^p) \right)^{\frac{1}{p}}.$
\end{definition}

\subsection{Vector valued $L^2$ spaces} 
\label{Vector valued spaces section}
Given a measure space $\Omega$ with a finite measure $\mu$ (a.k.a a finite measure space) and Banach space $V$, a function $s : \Omega \rightarrow V$ is called simple if it is of the form:
$$s(\omega) = \sum_{i=1}^n \chi_{E_i} (\omega) v_i,$$
where $\lbrace E_1,...,E_n \rbrace$ is a partition of $\Omega$ where each $E_i$ is a measurable set, $\chi_{E_i}$ is the indicator function on $E_i$ and $v_i \in V$. 

A function $f : \Omega \rightarrow V$ is called \textit{Bochner measurable} if it is almost everywhere the limit of simple functions, i.e., if there is a sequence of simple functions $s_n : \Omega \rightarrow V$ such that for almost every $\omega$, $f(\omega) = \lim_{n} s_n (\omega)$. Denote $L^2 (\Omega ; V)$ to be the space of Bochner measurable functions satisfying:
$$\Vert f \Vert_{L^2 (\Omega ; V)} = \left( \int_\Omega \Vert f (\omega) \Vert^2_V d \mu (\omega) \right)^{\frac{1}{2}} < \infty.$$ 

Given a bounded linear operator $T \in B(L^2 (\Omega, \mu))$, we can define a bounded linear operator $T \otimes id_V \in B(L^2 (\Omega ; V))$ by defining it first on simple functions and extending it to the whole space $L^2 (\Omega ; V)$.

We will also be interested in how $T \otimes id_V$ behaves under some operations - this is summed up in the following lemma:
\begin{lemma}
\label{L2 norm stability}
Let $(\Omega, \mu)$ be a measure space with a finite measure, $T$ a bounded operator on $L^2 (\Omega, \mu)$ and $C >0$ a constant. Let $\mathcal{B} = \mathcal{B} (C)$ be the class of Banach spaces defined as:
$$\mathcal{B} = \lbrace V : \Vert T \otimes id_V \Vert_{B(L^2 (\Omega ; V))} \leq C \rbrace.$$
Then this class is closed under quotients, subspaces, $l_2$-sums, and ultraproducts of Banach spaces, i.e., performing any of these operations on Banach spaces in $\mathcal{B}$ yields a Banach space in $\mathcal{B}$. Also, for any finite measure space $(\Lambda, \nu)$ and every $V \in \mathcal{B}$, we have that $L^2 (\Lambda ; V) \in \mathcal{B}$.
\end{lemma}

\begin{proof}
The fact that $\mathcal{B}$ is closed under quotients, subspaces and ultraproducts of Banach spaces was shown in \cite[Lemma 3.1]{Salle}. The fact that $\mathcal{B}$ is closed under $\ell_2$-sums is straight-forward and left for the reader (we will not make any use of it in this paper).

Let $(\Lambda, \nu)$ be a measure space with a finite measure and $V \in \mathcal{B}$. By our definition of vector valued spaces using simple functions, it is enough to check that the inequality holds for simple functions $s: \Omega \rightarrow L^2 (\Lambda ; V)$. Moreover, it is enough to check for simple functions  $s: \Omega \rightarrow L^2 (\Lambda ; V)$ whose values are simple functions in $L^2 (\Lambda ; V)$. In other words, if we identify $L^2 (\Omega ; L^2 (\Lambda ; V))$ with $L^2 (\Omega \times \Lambda ; V)$, we need to show that the needed inequality holds for functions of the form:
$$s(\omega, \lambda) = \sum_{i=1}^n \sum_{j=1}^m  \chi_{E_i} (\omega) \chi_{F_j} (\lambda) v_{i,j} ,$$
where $\lbrace E_1,...,E_n \rbrace$ is a measurable partition of $\Omega$, $\lbrace F_1,...,F_m \rbrace$ is a measurable partition of $\Lambda$, and $v_{i,j} \in V$.
Let $s$ be as above, then 
$$\Vert s \Vert^2 =  \sum_{i=1}^n \sum_{j=1}^m \mu (E_i) \nu (F_j) \Vert v_{i,j} \Vert_V^2.$$
We recall that for every measurable set $E \subseteq \Omega$ and every $v \in V$, the action of $T \otimes id_V$ on $\chi_E v$ is defined by
$$(T \otimes id_V)(\chi_E v) = T(\chi_E) v.$$
Similarly, for every function $f \in L^2 (\Lambda ; V)$, the action of $T \otimes id_{L^2 (\Lambda ; V)}$ on $\chi_E f$ is defined by 
$$(T \otimes id_{L^2 (\Lambda ; V)})(\chi_E f) = T(\chi_E) f.$$
Therefore, the action of $T \otimes id_{L^2 (\Lambda ; V)}$ on $s$ is as follows (we are abusing the notation; formally, the action of $T \otimes id_{L^2 (\Lambda ; V)}$ is defined on  $L^2 (\Omega ; L^2 (\Lambda ; V))$ and not on $L^2 (\Omega \times \Lambda ; V)$):
\begin{dmath*}
{(T \otimes id_{L^2 (\Lambda ; V)}) s = \sum_{i=1}^n  T(\chi_{E_i}) \sum_{j=1}^m \chi_{F_j} v_{i,j} = 
\sum_{j=1}^m \chi_{F_j} \sum_{i=1}^n  T(\chi_{E_i}) v_{i,j} =} \\
\sum_{j=1}^m \chi_{F_j} \sum_{i=1}^n  (T \otimes id_V) (\chi_{E_i} v_{i,j}) =
\sum_{j=1}^m \chi_{F_j} (T \otimes id_V) (\sum_{i=1}^n \chi_{E_i}  v_{i,j}).
\end{dmath*}
Note that written as above, for every $j$, $\sum_{i=1}^n (\chi_{E_i})  v_{i,j} \in L^2 (\Omega ; V)$ and therefore since $V \in \mathcal{E}$, we have for every $j$ that
$$\Vert  (T \otimes id_V) (\sum_{i=1}^n (\chi_{E_i})  v_{i,j}) \Vert^2 \leq C^2 \Vert \sum_{i=1}^n (\chi_{E_i})  v_{i,j} \Vert^2.$$
This yields that
\begin{dmath*}
\Vert (T \otimes id_{L^2 (\Lambda ; V)}) s \Vert^2 = 
\sum_{j=1}^m \nu (F_j)  \Vert (T \otimes id_V) (\sum_{i=1}^n (\chi_{E_i})  v_{i,j}) \Vert^2 \leq \\
\sum_{j=1}^m \nu (F_j) C^2 \Vert \sum_{i=1}^n (\chi_{E_i})  v_{i,j} \Vert^2 = 
C^2 \sum_{j=1}^m \nu (F_j) \sum_{i=1}^n \mu (E_i) \Vert v_{i,j} \Vert^2 = C^2 \Vert s \Vert^2,
\end{dmath*}
as needed.
\end{proof}

\subsection{Group representations on Banach spaces} 
Let $G$ be a locally compact group and $V$ a Banach space. Let $\pi$ be a representation $\pi :  G \rightarrow B (V)$, where $B(V)$ are the bounded linear operators on $V$. Throughout this paper we shall always assume $\pi$ is continuous with respect to the strong operator topology without explicitly mentioning it. We recall that given $\pi$, the dual representation $\pi^* : G \rightarrow B (V^*)$ is defined as 
$$\langle v, \pi^* (g) u  \rangle =  \langle \pi (g^{-1}) .v, u  \rangle, \forall g \in G, v \in V, u \in V^*.$$
We remark that $\pi^*$ might not be continuous for a general Banach space, but it is continuous for a large class of Banach spaces, called Asplund spaces defined below.

\subsection{Asplund spaces}
\begin{definition}
A Banach space $V$ is said to be an Asplund space if every separable subspace of $V$ has a separable dual.
\end{definition}

There are many examples of Asplund spaces - for instance every reflexive space is Asplund (see \cite{Yost} for an exposition on Asplund spaces). The reason we are interested in Asplund spaces is the following theorem of Megrelishvili:

\begin{theorem}\cite[Corollary 6.9]{Megre}
\label{Asplund implies continuous dual rep}
Let $G$ be a topological group and let $\pi$ be a continuous representation of $G$ on a Banach space $V$. If $V$ is an Asplund space, then the dual representation $\pi^*$ is also continuous.
\end{theorem}

Asplund spaces can be alternatively characterized as Banach spaces that have the Radon-Nikodym property (see definition in \cite{Yost}). Using this characterization it follows from a result of Sundaresan \cite{Sund} that the property of being Asplund is preserved when considering vector values $L^2$-spaces:
\begin{theorem}{\cite[Theorem 1]{Sund}}
\label{L^2 of Asplund is Asplund}
Let $V$ be a Banach space and let $(\Omega, \mu)$ be a measure space with a finite measure. Then $V$ is Asplund if and only if $L^2 (\Omega ; V)$ is Asplund.
\end{theorem}

\subsection{Group cohomology for groups acting on simplicial complexes}


Let $X$ be an $n$-dimensional simplicial complex and let $G$ be a group acting on $X$. Denote $X(k)$ to be the set of $k$-faces of $X$ and $\vec{X}(k)$ to be the set of ordered $k$-simplices of $X$. Let $V$ be a vector space and $\pi$ a representation of $G$ on $V$. 
Let $0 \leq k \leq n$ and let $\phi : \vec{X}(k) \rightarrow V$. Recall the following definitions:
\begin{itemize}
\item $\phi$ is anti-symmetric if for every permutation $\tau \in Sym \lbrace 0,...,k \rbrace$ and every $(v_{i_0},...,v_{i_k})$, $\phi ((v_{i_{\tau(0)}},...,v_{i_{\tau(k)}})) = \sgn (\tau) \phi ((v_{i_0},...,v_{i_k}))$.
\item $\phi$ is twisted by $\pi$, if for every $(v_{i_0},...,v_{i_k})$ and every $g \in G$, 
$$\pi (g)  \phi ((v_{i_0},...,v_{i_k})) = \phi (g.(v_{i_0},...,v_{i_k})).$$
\end{itemize}
 
For $0 \leq k \leq n$, denote $C^k (X, \pi)$ to be the space of maps $\phi : \vec{X}(k) \rightarrow V$ that are anti-symmetric and twisted by $\pi$. Define the differential map $d_k : C^k (X, \pi) \rightarrow C^{k+1} (X, \pi)$ in the usual way:
$$(d_k \phi) ((v_0,...,v_{k+1})) = \sum_{i=0}^{k+1} (-1)^i \phi ((v_0,..., \widehat{v_i},...,v_{k+1})).$$
As in the case of simplicial (untwisted) cohomology, we have that $d_{k+1} \circ d_k =0$ and $H^k (X,\pi ) = \Ker (d_k) / \im (d_{k-1})$. The next theorem states that under certain conditions, this cohomology is isomorphic to the group cohomology of $G$ with respect to the representation $\pi$:

\begin{theorem}\cite[X.1.12]{BorelW} 
\label{group cohomology using simplicial complex thm}
Let $G$ be a topological group and $X$ a contractible, locally finite simplicial complex. Assume that $G$ acts simplicially on $X$ and that this action is cocompact and proper. Assume further, that $V$ is a Banach space and $\pi$ is a continuous representation of $G$ on $V$, then $H^* (G,\pi) = H^* (X, \pi)$.
\end{theorem}

\section{Shapiro's Lemma}
\label{Shapiro's lemma sec}

\subsection{Framework}
\label{Framework subsec}
The aim of this section is to prove a version of Shapiro's Lemma. 

We fix the following notations: $X$ will denote an $n$-dimensional pure (i.e., every maximal cell is $n$-dimensional) contractible simplicial complex that is $(n+1)$-colorable (i.e., the vertices of $X$ can be colored by $n+1$ colors and every $n$-dimensional cell of $X$ has a vertex of every color) and locally finite (i.e., every vertex of $X$ is contained in a finite number of simplices). Throughout this section, $G$ will denote a locally compact, unimodular topological group with a Haar measure $\mu$, acting properly and cocompactly on $X$ such that the action preserves the coloring and $G$ acts transitively on the $n$-dimensional simplices of $X$ (note that this implies that $G$ is compactly generated). We denote by $\triangle$ a fixed $n$-dimensional simplex of $X$ that serves as the fundamental domain for the action of $G$. We denote by $\Gamma$ a countable subgroup of $G$ that also acts properly and cocompactly on $X$. So $\Gamma$ is a discrete cocompact subgroup of $G$. The case of interest for us is when $G=\mathbb{G}(K)$ - the $K$-points of a simple, $K$-rank $n$, $K$-algebraic group $\mathbb{G}$ when $K$ is a non-archimedean local field with a residue field of order $q$. In this case, $G$ acts properly on the Bruhat-Tits building $X$ associated with it, which is a contractible, pure $n$-dimensional, locally finite, $(n+1)$-colorable simplicial complex and the fundamental domain of the action of $G$ on $X$ is a single $n$-dimensional simplex. The thickness of $X$ is the minimal degree of all its $1$-dimensional links and it tends to infinity as $q$ tends to infinity. In this case, $\Gamma$ is a uniform (= cocompact) lattice, and by Margulis arithmeticity theorem (see \cite[Chapter 6]{Zimmer}), if $n \geq 2$, it is an arithmetic lattice.

\subsection{Shapiro's Lemma}

\begin{definition}
Let $G$ and $\Gamma$ as above, $V$ be a Banach space. Denote by $\nu$ the invariant measure on $G / \Gamma$ induced by the Haar measure of $G$. Define the Banach space 
$L^2 (G / \Gamma ; V)$ to be the space of Bochner measurable functions $f: G / \Gamma \rightarrow V$ with the norm:
\begin{equation}
\label{norm G/Gamma}
\Vert f \Vert = \left( \int_{G / \Gamma} \Vert f (g) \Vert_V^2 d \nu (g) \right)^{\frac{1}{2}}.
\end{equation}

By choosing a fundamental domain $D$ for the action of $\Gamma$ on $G$, we can identify functions $L^2 (G / \Gamma ; V)$ with  $L^2 (D ; V)$, where $D$ is taken with the restriction of the measure $\mu$.

Let $\pi$ be an isometric representation of $\Gamma$ on $V$. The induced representation of $\pi$ from $\Gamma$ to $G$, denoted $\Ind_{\Gamma}^G (\pi)_{L^2}$, is defined as follows:
\begin{dmath*}
\Ind_{\Gamma}^G (\pi)_{L^2} = {\lbrace f : G \rightarrow V }:  {\forall g \in G, h \in \Gamma, f (gh^{-1}) = \pi (h) f (g) \text{ and } f \in L^2 (G / \Gamma) \rbrace},
\end{dmath*}
where $f \in L^2 (G / \Gamma)$ means that $f$ is Bochner measurable when restricted to $D$ and with the norm defined in \eqref{norm G/Gamma} above. The reader should note that $\Vert f(g) \Vert$ is well defined on $G / \Gamma$, because $\pi$ is isometric and therefore for every $h \in \Gamma$, $g \in G$, $\Vert f (gh^{-1}) \Vert = \Vert \pi (h) f (g) \Vert  =  \Vert f(g) \Vert$.

Also, $G$ acts on $\Ind_{\Gamma}^G (\pi)_{L^2}$ by left translation, denoted $\lambda_{\Ind_{\Gamma}^G (\pi)_{L^2}}$, as: 
$$\lambda_{\Ind_{\Gamma}^G (\pi)_{L^2}}(g) f (g') = f( g^{-1} g'), \: \forall g,g' \in G.$$
\end{definition}

\begin{remark}
\label{equivalent definition}
The induced representation can also be defined as follows: 
define $\Ind_{\Gamma}^G (\pi)$ to be the vector space 
$$\Ind_{\Gamma}^G (\pi) = \lbrace f : G \rightarrow V \text{ continuous} : \forall g \in G, h \in \Gamma, f (gh^{-1}) = \pi (h) f (g) \rbrace,$$
and complete the vector space with respect to the $L^2$ norm as in the definition of $\Ind_{\Gamma}^G (\pi)_{L^2}$. The equivalence between these definitions is proven in \cite[Chapter 4]{Gaal} in the setting of isometric actions on Hilbert spaces, but the proof can be generalized to our setting. We will not make any use of this equivalent definition.
\end{remark}


\begin{proposition}
\label{norm of induced rep prop}
Let $G$, $\Gamma$, $V$ and $\pi$ be as above. Then $\Ind_{\Gamma}^G (\pi)_{L^2}$ is a Banach space and the action of $G$ on $\Ind_{\Gamma}^G (\pi)_{L^2}$ by left translation, denoted $\lambda_{\Ind_{\Gamma}^G (\pi)_{L^2}}$, is an isometric continuous representation of $G$ on $\Ind_{\Gamma}^G (\pi)_{L^2}$. 
\end{proposition}

\begin{proof}
The fact that $\lambda_{\Ind_{\Gamma}^G (\pi)_{L^2}}$ is isometric and continuous when $\pi$ is isometric is straight-forward and left for the reader. 
\end{proof}

Classically, Shapiro's Lemma is the equality $H^* (\Gamma , \pi) = H^* (G, \lambda_{\Ind_{\Gamma}^G (\pi)})$. This equality is proven in \cite{BorelW} for $\Ind_{\Gamma}^G (\pi)$ defined in Remark \ref{equivalent definition}, but not for $\Ind_{\Gamma}^G (\pi)_{L^2}$ which is a larger space (see Remark \ref{equivalent definition}). Below, we will prove the equality $H^* (\Gamma , \pi) = H^* (G, \lambda_{\Ind_{\Gamma}^G (\pi)_{L^2}})$ under the assumptions on $G$ and $\Gamma$ that are stated in the beginning of this section. We suspect that this equality is true even without our added assumptions, but in the case we are interested in, the proof that we give for this equality is direct and elementary.

\begin{theorem}[$L^2$-Shapiro's Lemma with coefficients in Banach representations]
\label{L^2 Shapiro's lemma}
Let $X$, $G$, $\Gamma$ be as above, $V$ a Banach space and $\pi$ an isometric representation of $\Gamma$ on $V$. Then
$$H^* (\Gamma, \pi) =  H^* (G, \lambda_{\Ind_{\Gamma}^G (\pi)_{L^2}}).$$
\end{theorem}

\begin{lemma}
\label{f-phi sigma lemma}
Let $X$, $G$, $\Gamma$ be as above, $V$ a Banach space and $\pi$ an isometric representation of $\Gamma$ on $V$. Given $\phi \in C^k (X, \pi)$ and $\sigma \in \vec{X}(k)$, define $f_{\phi, \sigma} : G \rightarrow V$ by 
$$f_{\phi, \sigma} (g) = \phi (g^{-1}.\sigma).$$ 
Then $f_{\phi, \sigma} \in \Ind_{\Gamma}^G (\pi)_{L^2}$ and if for some $g \in G$, $\phi (g^{-1}.\sigma) \neq 0$, then $\Vert f_{\phi, \sigma} \Vert >0$. 
\end{lemma}

\begin{proof}
We note that for every $h \in \Gamma$ and every $g \in G$, we have that
$$f_{\phi, \sigma} (g h^{-1}) = \phi (h g^{-1}.\sigma ) = \pi (h) \phi (g^{-1} . \sigma ) =  \pi (h) f_{\phi, \sigma} (g).$$
Hence, we are left to show that $f_{\phi, \sigma} \in \Ind_{\Gamma}^G (\pi)_{L^2}$ and that if for some $g \in G$, $\phi (g^{-1}.\sigma) \neq 0$, then $\Vert f_{\phi, \sigma} \Vert >0$. Most of the work in the rest of this proof is choosing a convenient fundamental domain $D$ for the action of $\Gamma$ on $G$.

By our assumptions, $\Gamma$ acts cocompactly on $X$ and therefore $\Gamma \setminus \vec{X} (k)$ is finite. In particular, there are $\sigma_1,...,\sigma_m \in \vec{X}(k)$ such that 
$$ \lbrace g. \sigma : g \in G \rbrace = \bigcup_{i=1}^m  \lbrace h. \sigma_i : h \in \Gamma \rbrace,$$
and the union above is disjoint. Fix $g_i \in G$, $i=1,...,m$, such that $g_i. \sigma = \sigma_i$ (such $g_i$'s exist, because we assumed that $G \setminus X$ is a single colored $n$-dimensional simplex). It follows that for every $g \in G$, there are $h \in \Gamma$ and a unique $i$ such that 
$$g^{-1}. \sigma = h . \sigma_{i} = h g_{i} . \sigma,$$
i.e., $g_{i}^{-1} h^{-1} g^{-1}. \sigma = \sigma$. If we denote the stabilizer of $\sigma$ in $G$ by $G_\sigma$, we deduce that there is $g_\sigma \in G_\sigma$ such that $g^{-1} = h g_{i} g_\sigma$ and so 
$$G = \bigcup_{i=1}^m G_\sigma g_{i}^{-1} \Gamma = \bigcup_{i=1}^m g_i^{-1} (g_i G_\sigma g_{i}^{-1}) \Gamma = \bigcup_{i=1}^m g_i^{-1} (G_{g_i.\sigma} ) \Gamma = \bigcup_{i=1}^m g_i^{-1} (G_{\sigma_i} ) \Gamma$$
 and this is a disjoint union. 
 
For every $i$, denote $\Gamma_{\sigma_i} =  G_{\sigma_i} \cap \Gamma$ and choose $D_{\sigma_i}$ to be a fundamental domain for the action of $\Gamma_{\sigma_i}$ on $G_{\sigma_i}$. We claim that $D=\bigcup_{i=1}^m g_i^{-1} (D_{\sigma_i} )$ is a fundamental domain for the action of $\Gamma$ on $G$. Indeed, $D_{\sigma_i} \subseteq G_{\sigma_i}$ and therefore $D$ is defined by a disjoint union and 
$$\left( \bigcup_{i=1}^m g_i^{-1} D_{\sigma_i} \right) \Gamma = \bigcup_{i=1}^m g_i^{-1} D_{\sigma_i}  \Gamma_{\sigma_i} \Gamma = \bigcup_{i=1}^m g_i^{-1} G_{\sigma_i}  \Gamma = G,$$
as needed.
 
With this choice of fundamental domain, it follows that 
\begin{dmath*}
\int_{G / \Gamma} \Vert  f_{\phi, \sigma} (g) \Vert_V^2 d \nu (g) = 
\int_{D} \Vert  f_{\phi, \sigma} (g) \Vert_V^2 d \mu (g) = 
\sum_{i=1}^m \int_{g_i^{-1} D_{\sigma_i}} \Vert  f_{\phi, \sigma} (g) \Vert_V^2 d \mu (g) = \\
\sum_{i=1}^m \int_{D_{\sigma_i}} \Vert  f_{\phi, \sigma} (g_i^{-1} g) \Vert_V^2 d \mu (g) = 
\sum_{i=1}^m \int_{D_{\sigma_i}} \Vert \phi (g^{-1} g_i.\sigma) \Vert_V^2 d \mu (g) = \\
\sum_{i=1}^m \int_{D_{\sigma_i}} \Vert \phi (g^{-1} .\sigma_i) \Vert_V^2 d \mu (g) = 
\sum_{i=1}^m \int_{D_{\sigma_i}} \Vert \phi (\sigma_i) \Vert_V^2 d \mu (g) = \\
\sum_{i=1}^m \mu (D_{\sigma_i}) \Vert \phi (\sigma_i) \Vert_V^2.
\end{dmath*}
Note that by the assumption of proper action of $G$ and of $\Gamma$ on $X$, we have for every $1 \leq i \leq m$, that $0 < \mu (G_{\sigma_i}) < \infty$ and $\Gamma_{\sigma_i}$ is a finite group. Hence, $0 < \mu (D_{\sigma_i}) < \infty$ and $f_{\phi, \sigma} \in \Ind_{\Gamma}^G (\pi)_{L^2}$. Also note that if for some $g \in G$, $\phi (g^{-1}.\sigma) \neq 0$, then there is $1 \leq i_0 \leq m$, $\phi (\sigma_{i_0}) \neq 0$ and therefore 
$$\int_{G / \Gamma} \Vert  f_{\phi, \sigma} (g) \Vert_V^2 d \nu (g) \geq \mu (G_{\sigma_{i_0}} / \Gamma_{\sigma_{i_0}}) \Vert \phi (\sigma_{i_0}) \Vert_V^2 >0.$$
\end{proof}

We can now prove Shapiro's Lemma in our setting:
\begin{proof}
By Theorem \ref{group cohomology using simplicial complex thm}, it is enough to prove that $H^* (X, \pi) = H^* (X, \lambda_{\Ind_{\Gamma}^G (\pi)_{L^2}})$. We will prove this by finding bijective linear maps $\Phi_k : C^k (X,\pi) \rightarrow C^k (X, \lambda_{\Ind_{\Gamma}^G (\pi)_{L^2}})$  for $k=0,...,n$ such that for every $\phi \in C^{k} (X,\pi)$, $d_k \Phi_k (\phi) = \Phi_{k+1} (d_k \phi)$. The existence of such maps shows that $H^* (X, \pi) = H^* (X, \lambda_{\Ind_{\Gamma}^G (\pi)_{L^2}})$ as needed.

Define $\Phi_k : C^k (X,\pi) \rightarrow C^k (X, \lambda_{\Ind_{\Gamma}^G (\pi)_{L^2}})$ by 
$$(\Phi_k (\phi)) (\sigma) = f_{\phi, \sigma},$$
where $f_{\phi, \sigma}$ is defined as in Lemma \ref{f-phi sigma lemma}. 

There are several thing we need to check. First, we need to check that for every $\phi \in C^{k} (X,\pi)$, it holds that $\Phi (\phi) \in C^k (X, \lambda_{\Ind_{\Gamma}^G (\pi)_{L^2}})$. By Lemma \ref{f-phi sigma lemma}, we have that $f_{\phi, \sigma} \in \Ind_{\Gamma}^G (\pi)_{L^2}$. By the definition of $f_{\phi, \sigma}$, it is also clear that $\Phi (\phi)$ is anti-symmetric since $\phi$ is anti-symmetric. Moreover, $\Phi_k (\phi)$ is also twisted by $\lambda_{\Ind_{\Gamma}^G (\pi)_{L^2}}$: let $g,g' \in G$, then
\begin{dmath*}
{\lambda_{\Ind_{\Gamma}^G (\pi)_{L^2}} (g). ((\Phi (\phi)) (\sigma)) (g') = 
((\Phi (\phi)) (\sigma)) (g^{-1} g') =
f_{\phi, \sigma} (g^{-1} g') = } \\
{\phi ((g^{-1} g')^{-1} . \sigma) =
\phi ((g')^{-1}. (g. \sigma)) =
f_{\phi, g. \sigma} (g') =
((\Phi (\phi)) (g.\sigma)) (g'),}
\end{dmath*}
as needed. Thus $\Phi_k : C^k (X,\pi) \rightarrow C^k (X, \lambda_{\Ind_{\Gamma}^G (\pi)_{L^2}})$ as claimed above. 

Second, we note that if $\phi \not\equiv 0$, then for some $\sigma$, $\phi (\sigma) \neq 0$ and therefore by Lemma  \ref{f-phi sigma lemma}, $\Phi (\phi) \not\equiv 0$ and therefore $\Phi$ is injective. 

Third, we will check that $\Phi$ is surjective. Let $\psi \in C^k (X, \lambda_{\Ind_{\Gamma}^G (\pi)_{L^2}})$, then for every $\sigma \in \vec{X} (k)$, $\psi (\sigma) \in \Ind_{\Gamma}^G (\pi)_{L^2}$. Since $\psi$ is twisted by $\lambda_{\Ind_{\Gamma}^G (\pi)_{L^2}}$, we have that for every $g \in G_\sigma$,
$$\psi (\sigma) = \psi (g. \sigma) = \lambda_{\Ind_{\Gamma}^G (\pi)_{L^2}} (g) \psi (\sigma).$$
The above equality is an equality in $\Ind_{\Gamma}^G (\pi)_{L^2}$, i.e., for almost every $g' \in G$, 
$\psi (\sigma) (g') = \lambda_{\Ind_{\Gamma}^G (\pi)_{L^2}} (g) \psi (\sigma) (g') =  \psi (\sigma) (g^{-1} g')$. In particular, there is $x_\sigma \in V$ such that for almost every $g \in G_\sigma$, $\psi (\sigma) (g) = x_\sigma$. Define $\phi_\psi : \vec{X} (k) \rightarrow V$, by $\phi_\psi (\sigma) = x_\sigma$, where $x_\sigma$ is as above. 

We will show that $\phi_\psi \in C^k (X, \pi)$ and that $\Phi (\phi_\psi) = \psi$. The fact that $\phi_\psi$ is anti-symmetric follows directly from the fact that $\psi$ is anti-symmetric. To see that $\phi_\psi$ is twisted by $\pi$, we note that for every $h \in \Gamma$ and every $\sigma \in \vec{X} (k)$, $x_{h. \sigma}$ was defined such that for almost every $g' \in G_{h. \sigma}$, $\psi  (h. \sigma) (g') = x_{h.\sigma}$. Note that $G_{h. \sigma} = h G_\sigma h^{-1}$, and therefore, for almost every $g \in G_\sigma$, $\psi (h. \sigma) (h g h^{-1}) = x_{h.\sigma}$. Thus,
\begin{dmath*}
{x_{h.\sigma} = \psi (h. \sigma) (h g h^{-1}) = \lambda_{\Ind_{\Gamma}^G (\pi)_{L^2}} (h) \psi (\sigma) (h g h^{-1}) =}
 \psi (\sigma) (g h^{-1})  = \pi (h) \psi (\sigma) (g),
\end{dmath*}
and since this holds for almost every $g \in G_\sigma$, it follows that
$$\phi_\psi (h. \sigma) = x_{h.\sigma} = \pi (h) x_{\sigma} = \pi (h) \phi_\psi (\sigma),$$
as needed.
To see that $\Phi (\phi_\psi) = \psi$, we will show that for almost every $g\in G$ and every $\sigma \in \vec{X} (k)$, $\Phi (\phi_\psi) (\sigma)(g) = \psi (\sigma) (g)$. We note that for almost every $g \in G$ and almost every $g' \in  G_\sigma$, $x_{g^{-1}.\sigma} = 
\psi (g^{-1}.\sigma) (g^{-1} g' g)$. Therefore, for almost every $g \in G$ and almost every $g' \in G_\sigma$, 
\begin{dmath*}
{\Phi (\phi_\psi) (\sigma)(g) = 
f_{\phi_\psi, \sigma} (g) =
\phi_\psi (g^{-1}.\sigma) = 
x_{g^{-1}.\sigma} = 
\psi (g^{-1}.\sigma) (g^{-1} g' g) =} \\
\lambda_{\Ind_{\Gamma}^G (\pi)_{L^2}} ((g')^{-1} g)  \psi (g^{-1}.\sigma) (g) = \psi (\sigma) (g).
\end{dmath*}

Last, one can easily see that $\Phi$ is linear and direct computation shows that for $\phi \in C^{k} (X,\pi)$, $d_k \Phi_k (\phi) = \Phi_{k+1} (d_k \phi)$.
\end{proof}

\section{Proof of Theorem \ref{Vanishing of cohomoloogy thm}}
\label{Vanishing of Cohomology sec}

Let $G=\mathbb{G}(K)$ be the $K$-points of a simple, $K$-rank $n$, $K$-algebraic group $\mathbb{G}$ when $K$ is a non-archimedean local field with a residue field of order $q$. The group $G$ acts on a Bruhat-Tits building $X$ and as noted above, the conditions of \cref{Framework subsec} for the action are fulfilled.  

Below, we will describe the main theorem of \cite{OppVanBan} (stated in the \S 1.2 of \cite{OppVanBan}) in the setting above. The setup of this theorem is as follows:
\begin{enumerate}
\item The class of Banach spaces $\mathcal{E}' = \mathcal{E}_3 (\mathcal{E}_2 (\mathcal{E}_1 (r,C_1),\theta_2),C_3)$ is derived by different types of deformations of Hilbert spaces with respect to any chosen constants  $r > 20$, $C_1 \geq 1$, $1 \geq \theta_2 >0, C_3 \geq 1$ (the constants determine the extent of the deformations) - for an exact definition see in \cite[\S 1.1]{OppVanBan}.
\item The class of Banach spaces $\mathcal{E} = \overline{\mathcal{E}'}$ is the closure of $\mathcal{E}'$ under quotients, subspaces, $l_2$-sums and ultraproducts.    
\end{enumerate} 
Under this setup, the main theorem of \cite{OppVanBan} can be stated as follows: 
\begin{theorem}
\label{main thm of VanBan}
For every choice of constants $r > 20$, $C_1 \geq 1$, $1 \geq \theta_2 >0, C_3 \geq 1$, there is a constant $Q = Q (r,C_1, \theta_2, C_3, n)$ such that if $q>Q$, then for every $V \in \mathcal{E} = \overline{\mathcal{E}'}$ and every continuous isometric representation $\rho$ of $G$ on $V$ such that $\rho^*$ is also continuous, $H^i (G, \rho) = 0$ for $i=1,...,n-1$.
\end{theorem}

The idea behind the proof of this Theorem is as follows. One fixes an $n$-dimensional simplex $\triangle$ in $X$ and defines a family of operators $T_\tau \in B(L^2 (G_{\tau}))$, where $\tau$ runs over the $(n-2)$-faces of $\triangle$. It is shown there that there is a constant $\varepsilon_0 >0$, such that for a given Banach space $V$, if $\Vert T_\tau \otimes id_V \Vert_{B(L^2 (G_{\tau} ; V))} \leq \varepsilon_0$, then for every isometric representation $\rho$ of $G$ on $V$, such that $\rho^*$ is continuous, $H^i (G, \rho) = 0$ for $i=1,...,n-1$. The class of Banach spaces $\mathcal{E}'$ is then defined in such a way that for a large enough $q$, $\Vert T_\tau \otimes id_V \Vert \leq \varepsilon_0$ for every $V \in \mathcal{E}'$ and every $\tau$. By \cite[Lemma 2.24]{OppVanBan}, passing to the closure does not change the bounds on $\Vert T_\tau \otimes id_V \Vert$ and therefore the cohomologies also vanish for $V \in \mathcal{E} = \overline{\mathcal{E} '}$.

We observe that if in the proof of Theorem \ref{main thm of VanBan} in \cite{OppVanBan}, we use Lemma \ref{L2 norm stability} of the current paper (instead of  \cite[Lemma 2.24]{OppVanBan}), we can extend the definition of the closure of $\mathcal{E}'$ in Theorem \ref{main thm of VanBan}. Define $\widetilde{\mathcal{E}} = \widetilde{\mathcal{E}'}$ to be the smallest class of Banach spaces that contains $\mathcal{E}'$, such that $\widetilde{\mathcal{E}}$ is closed under quotients, subspaces, $l_2$-sums, ultraproducts, \textbf{and}  such that for every finite measure space $(\Omega, \nu)$, if $V \in \widetilde{\mathcal{E}}$, then $L^2 (\Omega ; V) \in \widetilde{\mathcal{E}}$. By Lemma \ref{L2 norm stability}, if $\Vert T_\tau \otimes id_V \Vert \leq \varepsilon_0$ for every $V \in \mathcal{E}'$ and every $\tau$, then for every $V \in \widetilde{\mathcal{E}'}$,  $\Vert T_\tau \otimes id_V \Vert \leq \varepsilon_0$ for every $\tau$ and the rest of the proof of Theorem \ref{main thm of VanBan} is verbatim as in \cite{OppVanBan}. This yields the following:
\begin{theorem}
\label{extended main thm of VanBan}
For every choice of constants $r > 20$, $C_1 \geq 1$, $1 \geq \theta_2 >0, C_3 \geq 1$, there is a constant $Q = Q (r,C_1, \theta_2, C_3, n)$ such that if $q>Q$, then for every $V \in \widetilde{\mathcal{E}} = \widetilde{\mathcal{E}'}$ and every continuous isometric representation $\rho$ of $G$ on $V$ such that $\rho^*$ is also continuous, $H^i (G, \rho) = 0$ for $i=1,...,n-1$.
\end{theorem}

Next, we fix some constants $r', C_1', C_3'$ such that $r' > 20$, $C_1' \geq 1$, $ C_3' \geq 1$ (for example, we can take $r'=21, C_1' = C_3' = 1$). Given constants $1 < p_1 \leq 2 \leq p_2 < \infty$, we denote $\theta_{p_1,p_2} = \min \lbrace 2 - \frac{2}{p_1}, \frac{2}{p_2} \rbrace$.  With these notations, we define $\mathcal{E} (p_1,p_2)$ to be the the class of Banach spaces $\mathcal{E} (p_1,p_2) = \mathcal{E}_3 (\mathcal{E}_2 (\mathcal{E}_1 (r',C_1'),\theta_{p_1,p_2}),C_3')$. We will not repeat the definitions of the $\mathcal{E}_1, \mathcal{E}_2, \mathcal{E}_3$ here, but only recall that for such a choice, $\mathcal{E} (p_1,p_2)$ contains all $\theta_{p_1,p_2}$-strictly Hilbertian spaces (see \cite[\S 1.1.2]{OppVanBan}). In particular, for every $p_1 \leq p \leq p_2$, $(M_{k} (\mathbb{C}), \Vert . \Vert_p) \in \mathcal{E} (p_1,p_2)$. With this notation, we can prove Theorem \ref{Vanishing of cohomoloogy thm}:
\begin{theorem}
Let $G= \mathbb{G} (K)$ be a simple, $K$-rank $n$, $K$-algebraic group over a non-archimedean local field $K$ with a residue field of order $q$ and $\Gamma < G$ a cocompact lattice. For any $1 < p_1 \leq 2 \leq p_2 < \infty$, there exists a natural number $Q=Q(p_1,p_2,n)$ such that if $q > Q$, then $H^i  (\Gamma, V) =0$ for every $i=1,...,n-1$ and every Banach space of the form $V = \prod_{l \rightarrow \mathcal{U}} (M_{k_l} (\mathbb{C}), \Vert . \Vert_p)$ where $\mathcal{U}$ is any ultrafilter on $\mathbb{N}$ and $p_1 \leq p \leq p_2$.
\end{theorem}
 
\begin{proof}
For $1 < p_1 \leq 2 \leq p_2 < \infty$, applying Theorem \ref{extended main thm of VanBan} on $\mathcal{E} (p_1, p_2)$, there is a constant $Q = Q (p_1, p_2, n)$ such that if $q > Q$, then for every $V \in \widetilde{\mathcal{E} (p_1, p_2)}$ and every isometric representation $\rho$ of $G$ on $V$, if $\rho^*$ is continuous, then $H^i (G, \rho) = 0$ for $i=1,...,n-1$.

As noted above, for every $k \in \mathbb{N}$ and every $p_1 \leq p \leq p_2$, $(M_{k} (\mathbb{C}), \Vert . \Vert_p) \in \mathcal{E} (p_1,p_2)$. Therefore for any choice of ultrafilter $\mathcal{U}$ on $\mathbb{N}$, 
$V = \prod_{l \rightarrow \mathcal{U}} (M_{k_l} (\mathbb{C}), \Vert . \Vert_p)$ is in $\widetilde{\mathcal{E} (p_1,p_2)}$. By \cite[Corollary 5.3]{PX}, every space $(M_{k} (\mathbb{C}), \Vert . \Vert_p)$ is uniformly convex and the bound on the modulus of convexity depends only on $p$. As a result, $V = \prod_{l \rightarrow \mathcal{U}} (M_{k_l} (\mathbb{C}), \Vert . \Vert_p)$ is uniformly convex and, by Milman-Pettis theorem (see for instance \cite[Theorem 5.2.15]{Megginson}), $V$ is reflexive and thus Asplund. 

Given an isomeric representation $\pi$ of $\Gamma$ on $V$, $V'=\Ind_{\Gamma}^G (\pi)_{L^2}$ is isometrically isomorphic to the Banach space $L^2 (D, \mu ; V)$ where $D$ is a fundamental domain of $G/ \Gamma$. Thus, $V' \in \widetilde{\mathcal{E} (p_1,p_2)}$ and by Theorem \ref{L^2 of Asplund is Asplund}, $V'$ is Asplund. The induced representation $\lambda_{\Ind_{\Gamma}^G (\pi)_{L^2}}$ is a continuous isometric representation on $V'$ (which is Asplund) and therefore by Theorem \ref{Asplund implies continuous dual rep}, $\lambda_{\Ind_{\Gamma}^G (\pi)_{L^2}}^*$ is continuous and the conditions of Theorem \ref{extended main thm of VanBan} hold. As a result, $H^i (G, \lambda_{\Ind_{\Gamma}^G (\pi)_{L^2}}) = 0$ for $i=1,...,n-1$, and by Theorem \ref{L^2 Shapiro's lemma} (which is our version of Shapiro's Lemma) it follows that $H^i (\Gamma, \pi) = 0$ for every $i=1,...,n-1$.
\end{proof}

\bibliographystyle{plain}
\bibliography{bibl}

\end{document}